\documentclass[12pt]{amsart}
\usepackage{amsmath}
\usepackage{amssymb}
\usepackage{eucal}
\usepackage{amsthm}
\usepackage{amscd}
\usepackage{hyperref}
\usepackage{url}

\usepackage{amsfonts}
\usepackage[dvipsnames, usenames]{color}
\usepackage{hyperref}
\usepackage{manfnt}
\usepackage{mathrsfs}
\usepackage{verbatim}

\numberwithin{equation}{section}

\newtheorem{thm}{Theorem}[section]
\newtheorem{cor}{Corollary}[section]

\newtheorem{prop}{Proposition}[section]

\newtheorem{defn}{Definition}[section]

\newtheorem{question}{Question}[section]

\newcommand{\An}[1]{\mathbb{A}^{#1}}

\newcommand{\cF}{\mathcal{F}}

\newcommand{\cU}{\mathcal{U}}

\newcommand{\nfracxpx}{\frac{x'}{x}}
\newcommand{\nfracx}{\frac{1}{x}} 

\newcommand{\m}{\mathfrak{m}}

\newcommand{\monex}{\mathfrak{m}\{1/x\}}
\newcommand{\mx}{\mathfrak{m}\{x\}}

\newcommand{\Pn}[1]{\mathbb{P}^{#1}}
\newcommand{\Po}{xx''=x'} 

\newcommand{\Q}{\mathbb{Q}}

\usepackage{amssymb,}
\usepackage[]{amsmath, amsthm, amsfonts,graphicx, amscd,}
\usepackage[all,cmtip]{xy}
\input amssym.def \input amssym

\begin{document}
\title[Complete differential varieties]{New examples (and counterexamples) of complete finite-rank differential varieties}
\author{William D. Simmons }
\thanks{The present paper is based on the author's doctoral thesis under the supervision of Professor David Marker. I would like to thank Dave as well as other colleagues and friends with whom I have discussed the differential completeness problem.}
\address{wsimmo@sas.upenn.edu \\
Department of Mathematics\\
University of Pennsylvania\\
David Rittenhouse Laboratory\\
209 South 33rd Street\\
Philadelphia, PA 19104-6395
}

\maketitle

\begin{abstract}
Differential algebraic geometry seeks to extend the results of its algebraic counterpart to objects defined by differential equations. Many notions, such as that of a projective algebraic variety, have close differential analogues but their behavior can vary in interesting ways. Workers in both differential algebra and model theory have investigated the property of completeness of differential varieties. After reviewing their results, we extend that work by proving several versions of a ``differential valuative criterion" and using them to give new examples of complete differential varieties. We conclude by analyzing the first examples of  incomplete, finite-rank projective differential varieties, demonstrating a clear difference from  projective algebraic varieties.
 \end{abstract}

\begin{section}{Introduction}
The fundamental theorem of elimination theory states that projective algebraic varieties over an algebraically closed field $K$ are \emph{complete}: If $V$ is such a variety and $W$ is an arbitrary variety over $K$, then the second projection map $\pi_{2}:V\times W\rightarrow W$ takes Zariski-closed sets to Zariski-closed sets. This property is tightly linked to projectiveness, as shown by the example of the affine hyperbola $xy-1=0$. The images of the projections to either axis lack 0 but contain every other point of the affine line. We must ``close up" the variety with a point at infinity to ensure a closed projection.

The geometric picture is clear and standard algebraic tools readily prove the theorem (see, for instance, \cite{shaf}). However, this is not the case with the same question in differential algebraic geometry.

A \emph{differential ring} (or $\Delta$-ring) is a commutative ring $R$ with 1 and a finite set of maps $\Delta$ such that for each $\delta\in \Delta$ (\emph{derivations} on $R$) and $x,y\in R$, $\delta(x+y)=\delta(x)+\delta(y)$ and $\delta(xy)=\delta(x)y + x\delta(y)$. 
We consider differential fields of characteristic zero with a single derivation $\delta$ as well as differential polynomial rings with coefficients belonging to a differential field. (To indicate that we are in the ordinary, rather than partial, case we write $\delta$-ring in place of $\Delta$-ring.) 

Over a differential field $F$, differential polynomial equations define closed sets (differential varieties) in the \emph{Kolchin topology}, a differential analogue of the Zariski topology on $\mathbb{A}^{n}(F)$ and $\mathbb{P}^{n}(F)$. If $I$ is a set of of differential polynomials, we write $\mathbf{V}(I)$ for the corresponding zero locus.  Given $X$ a subset of $\mathbb{A}^{n}(F)$ or $\mathbb{P}^{n}(F)$ we use the notation $\overline{X}$ to denote the Kolchin closure of $X$ (i.e., the intersection of all Kolchin-closed sets containing $X$). In projective space over $F$, the zero locus is only well defined if the corresponding differential polynomials are $\delta$-homogeneous; i.e., there is a natural number $d$ such that for all $\lambda \in F\setminus 0$, $f(\lambda \bar{x})=\lambda^{d}(f(\bar{x}))$ . A $\delta$-polynomial in the differential polynomial ring $F\{x_{0}, \dots,x_{n-1}\}$ can be $\delta$-homogenized by replacing each $x_{i}$ with $x_{i}/x_{n}$, applying the quotient rule to compute any derivatives that appear, and clearing denominators. For instance, $x_{0}'\in F\{x_{0}\}$ becomes $x_{1}x_{0}'-x_{0}x_{1}'$. When we refer to the projective closure of an affine $\delta$-variety, we mean the intersection of all projective $\delta$-varieties containing the affine one; in general the differential projective closure is a proper subset of the usual algebraic projective closure.

\begin{defn}
Let $V$ be a projective $\delta$-variety defined over a differential field $F$. We say $V$ is \emph{differentially complete} (or \emph{$\delta$-complete}) if for every projective $\delta$-variety $W$ over $F$ and Kolchin-closed subset $Z$ of $V\times W$, the second projection map $\pi_{2}: V\times W \rightarrow W$ sends $Z$ to a Kolchin-closed set.
\end{defn}

\noindent Closedness is a local property, so as in the non-differential case we may replace $W$ with affine space $\mathbb{A}^{n}$.

Differential algebraic geometry is much helped by a good analogue for algebraically closed fields, and we find one in the differentially closed fields of characteristic zero. These objects, first studied by model theorists, are the models of a first-order theory $DCF_{0}$ (in the case of $m$ commutating derivations, the corresponding theory is $DCF_{0,m}$). If $\mathcal{F}$ is such a field, we write $\mathcal{F}\models DCF_{0}$. Differentially closed fields are a natural choice for studying $\delta$-completeness. In particular they are \emph{existentially closed}, implying that if a system of differential polynomial equations with coefficients from a model of $DCF_{0}$ is solvable in an extension $\delta$-field, then it is already solvable in the base field.

Our problem, first looked at in essentially this form by E.R. Kolchin \cite{MR0352067}  and later by W.Y. Pong \cite{PongDiffComplete2000}, is this: 
\begin{question}
If $V$ is a projective $\delta$-variety defined over $F\models DCF_{0}$, is $V$ differentially complete? 
\end{question}

The answer is ``not necessarily", but Pong showed one reason for the counterexamples: $\delta$-completeness requires a variety to have \emph{finite rank} in any of the usual model-theoretic or differential algebraic senses. (See \cite{Marker:model_thy_of_fields} or \cite{marker:msri} for an overview. In the differential algebra literature the notion of differential dimension is usually taken to mean differential transcendence degree of the differential function field over the base field; being zero-dimensional is equivalent in the ordinary case to having finite Morley rank, etc.) In this paper we take the basic setup used by Pong and further develop it into an algorithmic strategy for attacking the $\delta$-completeness problem. Using this approach, we verify new examples of complete $\delta$-varieties over $DCF_{0}$. Importantly, we also analyze the first example of an incomplete finite-rank projective $\delta$-variety (introduced by the author in \cite{SimmonsThesis}); hence we know that this class is not empty.

In the next section we summarize the literature and major results around differential completeness. Refer to \cite{PongDiffComplete2000}, \cite{SimmonsThesis} for additional details.
\end{section}

\begin{section}{Further background}

The research thread we add to in this paper has two main strands. The first is differential algebraic and is comprised mainly of the papers \cite{MR0352067} of Kolchin, \cite{BlumComplete},\cite{BlumExtensions} of P. Blum, and \cite{MorrisonSD},\cite{morrison79} of Morrison. The second strand incorporates model-theoretic ingredients. In particular we need a note \cite{vandenDries} by van den Dries on positive quantifier elimination and Pong's work on differential completeness \cite{PongDiffComplete2000}.
 
Kolchin (\cite{MR0352067}) examined differential varieties in projective space and used the property of differential completeness to extend classical results about linear dependence over algebraic varieties. He also studied the notion for its own sake, showing that over a close analogue of differentially closed fields the projective closure of the constants (i.e., elements whose derivatives are all zero) is differentially complete and proving that various (infinite-rank) varieties like $\mathbb{P}^{n}$ are not differentially complete.
In \cite{PongDiffComplete2000} Pong gave explicit equations for one of the counterexamples. We explain it here in order to later contrast with our example of an incomplete finite-rank projective $\delta$-variety.

Define a Kolchin-closed subset $V$ of $\mathbb{A}^{1}\times\mathbb{A}^{1}$ with two $\delta$-polynomial equations
\[
  y(x')^{2} + x^{4}-1=0,\,\, 2yx'' + y'x' +4x^{3}=0,
\]

\noindent where $x$ corresponds to the first factor and $y$ to the second. Note that the point at infinity $(1:0)$ does not satisfy the $\delta$-homogenizations with respect to $x$ of these equations for any value of $y$, so $V$ is actually closed in $\mathbb{P}^{1}\times\mathbb{A}^{1}$. The second equation is the derivative of the first, followed by division by $x'$. It follows, using also the fact that $V$ is defined over a differentially closed field, that the image of the second projection consists of all points in $\mathbb{A}^{1}$ except 0. By irreducibility of $\mathbb{A}^{1}$ (with respect to the Kolchin topology; see Theorem B.1 in \cite{Marker:model_thy_of_fields} for a proof), the image is not Kolchin-closed. 

In a different direction, Blum and Morrison investigated differential analogues of valuation rings. The central object for our purposes is a \emph{maximal differential ring}.

\begin{defn}
Let $R$ be a $\delta$-subring of a $\delta$-field $K$. Let $f:R\rightarrow L$ be a $\delta$-ring homomorphism into a $\delta$-field $L$. (A $\delta$-ring homomorphism is a ring homomorphism $f$ from a $\delta$-ring $R_{1}$ into $\delta$-ring $R_{2}$ such that $f(\delta_{R_{1}}(x))=\delta_{R_{2}}(f(x))$ for all $x\in R_{1}$. The symbols $\delta_{R_{1}}$ and $\delta_{R_{2}}$ denote the derivations of $R_{1}$ and $R_{2}$, respectively.) We say $R$ is a maximal $\delta$-ring (with respect to $K$, or simply $K$-maximal) if $f$ does not properly extend to a $\delta$-ring homomorphism with domain a larger $\delta$-subring of $K$ and codomain a $\delta$-field.
\end{defn}

The crucial fact (\cite{BlumExtensions}) is that a $K$-maximal $\delta$-ring $R$ is a local differential ring with unique maximal $\delta$-ideal $\m$ (i.e., $\m$ is the unique maximal ideal of $R$, and $\m$ is also closed under the derivation). Being a $K$-maximal $\delta$-ring is not as strong as being a valuation ring. Nonetheless, such rings have the useful property that $x\in K\setminus R$ if and only if $1\in \mx$, where $\mx$ is the differential subring of $K$ generated by $\m$ and $x$.

Three additional results due to Morrison are critical to our work on $\delta$-completeness; they are Proposition 4 and its corollary in \cite{MorrisonSD} and Corollary 3.2 in \cite{morrison79}.

\begin{thm}
Let $R$ be a maximal differential ring of a differential field $K$, and let $x$ be an element of $K$. Suppose that $\nfracx$ is not in $R$ and that there is a linear relation 

\[
x^{(k)} =\sum_{i=0}^{k-1} a_{i}x^{(i)} +r,
\] 

\noindent where $r$ and $a_{i}\text{ }(0\leq i\leq k-1)$ are elements of $R$. Then $x\in R$.
\end{thm}

\begin{cor}
Let $R$ be a maximal differential ring of a differential field $K$, and $x$ a non-zero element of $K$. If for some positive integer $k$ either $x^{(k)}$ or $x^{(k)}/x$ is in $R$, then either $x$ or $\nfracx$ is in $R$. 
\end{cor}

\begin{thm}
If $\phi_{0}:R\rightarrow \Omega$ is a differential specialization and if $x$ is integral over $R$, then $\phi_{0}$ extends differentially to $x$. Thus a maximal differential ring of a differential field $L$ is integrally closed in $L$.
\end{thm}

(A differential specialization is a non-zero $\Delta$-homomorphism from a $\Delta$-ring containing $\Q$ into a $\Delta$-field.)

Now we turn to the model-theoretic strand. Van den Dries observed the following ``positive 	quantifier elimination" result (\cite{vandenDries}):  

\begin{thm} 
Let $T$ be an $\mathscr{L}$-theory and $\varphi(\bar{x})$ an $\mathscr{L}$-formula. Then $\varphi$ is equivalent modulo $T$ to a positive (i.e., lacking negations) quantifier-free formula if and only if for all $\mathcal{M},\mathcal{N} \models T$, substructures $A\subseteq \mathcal{M}$, $\bar{a}\in A$, and $\mathscr{L}$-homomorphisms $f:A\rightarrow \mathcal{N}$ we have $\mathcal{M}\models \varphi(\bar{a}) \implies \mathcal{N}\models \varphi(f(\bar{a}))$.
\end{thm}

\noindent (The statement refers to ``weak homomorphisms"; i.e., we only assume that for a relation symbol $R$, $\mathcal{M}\models R(\bar{a})$ implies $\mathcal{N}\models R(f(\bar{a}))$ but not necessarily the reverse.) Van den Dries noted that by letting $T$ be the theory of algebraically closed fields and using basic facts about valuation rings, the result gives an easy proof of the fundamental theorem of elimination theory. 

In \cite{PongDiffComplete2000} Pong extended Kolchin's work and gave a test for $\delta$-completeness using van den Dries' positive quantifier elimination. Pong proved that every $\delta$-complete $\delta$-variety has finite rank. Intuitively, this fact means that differential varieties defined by underdetermined systems of equations always ``have enough room" to hold counterexamples like the one cited earlier for $\mathbb{P}^1$.  Pong also showed that any $\delta$-complete $\delta$-variety is isomorphic to a $\delta$-subvariety of $\An{1}$. Similarly, any finite-rank $\delta$-variety, affine or projective, maps via an injective $\delta$-variety morphism into $\An{1}$. However, unless a given variety is known beforehand to be complete, we cannot say \emph{a priori} that the image is closed in $\An{1}$. 
Using the algebraic results of Blum and Morrison to compensate for maximal differential rings' failure to be valuation rings in general, Pong applied the positive quantifier elimination criterion to prove:

\begin{thm} 
Let $V\subseteq \mathbb{A}^{n}$ be an affine $\delta$-variety defined over $DCF_{0}$.  Then the following are equivalent:

\begin{enumerate}
\item $V$ is $\delta$-complete.

\item For any $K$-maximal $\delta$-ring $R$, we have $V(K)=V(R)$ (i.e., if a tuple $\bar{a}\in \mathbb{A}^{n}$ has all of its coordinates in $K$ and $\bar{a}\in V$, then every coordinate of $\bar{a}$ belongs to $R$). 

\end{enumerate}•
\end{thm}

Note that the hypothesis of finite rank does not appear, so the algebraic property of the $K$-points of $V$ descending to $R$ interestingly guarantees finite rank. Pong used this affine ``valuative criterion" to produce new examples of $\delta$-complete varieties. In particular, he proved completeness of the projective closure of any $\delta$-variety in $\An{1}$ defined by $x'=P(x)$, where $P$ is a non-differential polynomial (if $P$ is 0, we recover the single-variable case of Kolchin's example of the constants). 

For the rest of this paper we follow Pong's general conventions in \cite{PongDiffComplete2000} when working with $\delta$-varieties:

\begin{enumerate}
\item All $\delta$-varieties considered are defined by $\delta$-polynomials with coefficients from a fixed arbitrary field $\mathcal{F}\models DCF_{0}$. In particular, completeness means completeness with respect to varieties defined over $\mathcal{F}$. 
\item The first-order language $\mathscr{L}$ we use is the language of differential rings, augmented with constant symbols for each element of $\cF$. Our basic theory is $Th_{\mathscr{L}}(\cF)$; i.e., the elementary diagram of $\cF$. This ensures that all models of our theory contain an isomorphic copy of $\cF$.       

\item We assume that $\cF$ is contained in a large saturated model $\cU$ of $DCF_{0}$. Unless stated otherwise, $\Pn{n}$ and $\An{m}$ mean $\Pn{n}(\cU)$ and $\An{m}(\cU)$. We consider a $\delta$-variety $V$ to be comprised of all the $\cU$-points satisfying the defining $\delta$-polynomial equations of $V$.
\end{enumerate}•

Before moving on we should mention several more recent papers addressing similar questions. Prestel (\cite{prestel:val}) uses methods closely related to Pong's in order to prove completeness of projective varieties (non-differential) over algebraically closed and real closed valued fields. Pillay applies more abstract model theory in \cite{PillayDvar} to prove a completeness result about algebraic $D$-groups (a category closely related to that of finite-rank differential varieties).  Freitag translates Pong's paper in \cite{DeltaCompleteness} to the setting of $DCF_{0,m}$ and finds that most results transfer over with little change. Finally, in \cite{JamesOmarWilliam} Freitag, Le\'{o}n S\'{a}nchez, and the present author study a possible simplification of the $\delta$-completeness problem as well as a result about generalized Wronskians that involves $\delta$-complete varieties. The proposed simplification (which is valid assuming a conjecture from differential algebra) is that in evaluating differential completeness of a $\Delta$-variety over $DCF_{0,m}$, it suffices to consider only second factors that have finite rank. That is, if $V$ is a projective $\Delta$-variety, then $V$ is not $\Delta$-complete if and only if there is a finite-rank affine $\Delta$-variety $W$ and $\Delta$-closed subset $Z$ of $V\times W$ whose image in $W$ under projection is not $\Delta$-closed.

\end{section}

\begin{section}{Modified valuative criteria}

We have seen that Pong used his differential valuative criterion to demonstrate $\delta$-completeness of $\delta$-varieties in $\Pn{1}$ whose restrictions to $\An{1}$ are defined by equations of the form $x'=P(x)$. Can this technique be extended to prove completeness results for larger classes of finite-rank $\delta$-varieties? The answer is yes, but it is desirable to modify the valuative criterion before undertaking the project.

First, the valuative criterion only applies to $\delta$-varieties in $\An{n}$. It is true that finite-rank varieties miss generic hyperplanes (\cite{PongDiffComplete2000}), so we may assume that up to projective equivalence a given finite-rank projective $\delta$-variety is contained in a single standard affine chart. But what do such varieties look like? Even in $\An{1}$, it can be non-obvious whether a specific affine $\delta$-variety is projective. While studying complete differential varieties we found the following anomalous example: The projective closure of the affine differential variety $x''=x^2$ does not contain the point at infinity even though the point at infinity satisfies the $\delta$-homogenization of $x''-x^2$. (See the appendix for details.)  

Second, the author's experience in generalizing Pong's work indicates that it is only necessary to treat the elements of $\m$ as differential indeterminates. The calculations boil down to formal elimination arguments. This suggests that it would be more appropriate to have a ``syntactic" version of the valuative criterion that deals with polynomials instead of maximal $\delta$-rings. 

In this section we generalize the valuative criterion to $\delta$-subvarieties of $\mathbb{P}^{n}$ and give two alternate versions having, respectively, more geometric and computational content.

Our proof strategy is the same as Pong's, but some care is required to make the argument work in projective space. As our new examples of complete $\delta$-varieties depend on this result, we lay out the details. First we make an important definition.

\begin{defn}
Let $p = (p_{0}:p_{1}:\dots :p_{n}) \in \Pn{n}$ and let $S$ be a $\delta$-ring. With some abuse of terminology, we say $p$ is in $S$ (denoted $p\in S$) if for some $0\leq i\leq n$ we have $p_{i}\neq 0$ and $\frac{p_{j}}{p_{i}}\in S$ for all $0\leq j \leq n$. If $V$ is a differential subvariety of $\mathbb{P}^{n}$ such that $p\in V$ and $p\in S$, we write $p\in V(S)$.
\end{defn}

Notice that $p$ being in $S$  is independent of the choice of homogeneous coordinates for $p$, so this definition is reasonable. It is crucial to note that it is not sufficient for $p$ to have a representative such that every coordinate belongs to $S$.  More particularly, $p\in S$ implies that $p$ has a representative such that at least one coordinate is equal to $1$ and every other coordinate also belongs to $S$. The motivation for this stronger definition is that it allows us to state and prove a valuative criterion for projective differential varieties in a way that mirrors the affine case.  

\begin{thm} 
Let $V$ be a differential subvariety of $\mathbb{P}^{n}$.  Then $V$ is $\delta$-complete if and only if for every $K$-maximal $\delta$-ring $(R,f,\m)$ and point $p\in V(K)$ we have $p\in V(R)$. 
\end{thm}
\begin{proof}

\noindent To prove the forward direction we establish the contrapositive. Suppose $V$ has a point $p$ such that $p\in K$ but $p\notin R$. We show that $V$ is not $\delta$-complete by finding a Kolchin-closed set of $V\times \An{m}$ whose projection is not Kolchin closed.

By assumption there is a point $p=(p_{0}:\dots:p_{n})\in V(K)$ such that $p\notin R$; our notion of $p\in K$ is independent of representative, so we may assume each coordinate of $p$ belongs to $K$. Let $I\subseteq \{0,\dots,n\}$ denote the set of indices of non-zero coordinates of $p$. The fact that $p\notin R$ implies for each $i\in I$ there is an index $j_{i}$ such that $\frac{p_{j_{i}}}{p_{i}}\notin R$. (Necessarily $j_{i}\neq i$ because $\frac{p_{i}}{p_{i}}=1\in R.$) It follows from the properties of maximal $\delta$-rings that $1\in \m\{\frac{p_{j_{i}}}{p_{i}}\}$, so there exist elements $m_{i_{k}}\in \m$ satisfying an equation $\sum_{k} m_{i_{k}}t_{i_{k}}\left(\frac{p_{j_{i}}}{p_{i}}\right) =1$, where the $t_{i_{k}}$ are $\delta$-monomials in $\frac{p_{j_{i}}}{p_{i}}$.

Let $P_{i}(\bar{x},\bar{y})$ be the $\delta$-polynomial resulting from clearing denominators in \\$\sum_{k} y_{i_{k}}t_{i_{k}} \left(\frac{x_{j_{i}}}{x_{i}}\right) -1$, where $\bar{x},\bar{y}$ are differential indeterminates. Note that $P_{i}$ is $\delta$-homogeneous in $\bar{x}$ and so defines a $\delta$-subvariety of a product variety where the $\bar{x}$ are projective coordinates and the $\bar{y}$ are affine coordinates. Also observe that every $\bar{x}$-monomial in $P_{i}$ has a coefficient from $\bar{y}$ with the exception of one: the monomial obtained by multiplying $-1$ by a positive power of $x_{i}$ when clearing denominators. 
 By design, the following formula $\varphi(\bar{y})$ is true in $K$ when $y_{i_{k}}$ is interpreted as $m_{i_{k}}$:
\[
\exists \bar{x}\left( \bar{x}\in V \wedge \left(\wedge_{i\notin I}x_{i} =0\right)\wedge \left(\vee_{i\in I} x_{i}\neq 0\right) \wedge \left( \wedge_{i\in I} P_{i}(\bar{x},\bar{y})=0\right) \right).
\]

\noindent (That is, $p$ witnesses the truth of $\varphi(\bar{m}$).) The formula $\varphi(\bar{y})$ defines a projection whose failure to be closed would imply that the projective $\delta$-variety $V\cap \wedge_{i\notin I} (x_{i}=0)$ is not complete. By van den Dries' positive quantifier-elimination test, we simply need to verify that $\varphi(f(\bar{m}))$ is not true in the codomain of $f$.

Because $\m=\text{ker}(f)$, for every $i\in I$ all monomials of $P_{i}(\bar{x}, f(\bar{m}))$ vanish except for $-x_{i}^{r_{i}}$ for some positive integer $r_{i}$. Then $\varphi(f(\bar{m}))$ asserts that  
\[
\exists \bar{x}\left( \bar{x}\in V \wedge \left(\wedge_{i\notin I}x_{i} =0\right)\wedge \left(\vee_{i\in I} x_{i}\neq 0\right) \wedge \left( \wedge_{i\in I} (-x_{i}^{r_{i}}=0)\right) \right),
\]

\noindent which is contradictory. This proves that $V$ is not complete. 

For the reverse direction, let a collection $\{P_{j}(\bar{x},\bar{y})\}$ of $\delta$-polynomials define an arbitrary subvariety of $V \times \mathbb{A}^{m}$, where $\bar{x},\bar{y}$ are respectively projective and affine coordinates. To prove $V$ is $\delta$-complete we must show that the following $\varphi(\bar{y})$ is equivalent to a positive quantifier-free formula:
\[
\exists \bar{x}\left( \bar{x}\in V \wedge \left(\vee_{0\leq i\leq n} (x_{i}\neq 0)\right) \wedge \left( \wedge_{j} P_{j}(\bar{x},\bar{y})=0\right) \right).
\]

\noindent It is enough to prove that if $(R,f,\m)$ is a $K$-maximal differential ring such that $f:R\rightarrow L$ and $K\models \varphi (\bar{a})$ for elements $\bar{a}$ of $R$, then $L \models \varphi(f(\bar{a}))$. 

Let $\psi(\bar{x},\bar{y})$ be the subformula of $\varphi$ such that $\varphi(\bar{y})= \exists \bar{x}\psi(\bar{x},\bar{y})$. Since $K\models \varphi(\bar{a})$, there is $p= (p_{0}:\dots:p_{n})\in K$ such that $K \models \psi(p_{0},\dots,p_{n}, \bar{a})$. (What $K\models \varphi(\bar{a})$ actually tells us is that each $p_{i}\in K$ and some $p_{i}\neq0$, but $K$ is a field so we may divide by any non-zero coordinate and still have every coordinate in $K$; thus the definition of $p\in K$ is satisfied.) We have $p\in V(K)$, so it follows by hypothesis that $p\in V(R)$. Dividing by coordinate $p_{i}$ for some $0\leq i\leq n$ ensures that every resulting coordinate belongs to $R$; importantly, the resulting $i$-th coordinate is 1.  
 
Hence $\psi(\frac{p_{0}}{p_{i}},\dots, 1,\dots ,\frac{p_{n}}{p_{i}}, \bar{a})$ has parameters from $R$ and is true in $K$. The map $f$ is a homomorphism with domain $R$, homomorphisms preserve satisfaction of positive quantifier-free formulas, and the only non-positive subformula of $\psi(\bar{x},\bar{y})$ is the disjunction $\vee_{0\leq i\leq n}(x_{i}\neq 0)$, so we conclude $L \models \psi(f(\frac{p_{0}}{p_{i}}),\dots,f(1)=1 , \dots,f(\frac{p_{n}}{p_{i}}), f(\bar{a}))$. This implies the desired $L \models \varphi(f(\bar{a}))$, finishing the proof.

\end{proof}

For future reference, we call this result the \emph{projective valuative criterion}. For $\delta$-subvarieties of $\mathbb{P}^{1}$ the theorem boils down to the form we use for concrete examples in the next section:

\begin{thm}
Let $V$ be a differential subvariety of $\mathbb{P}^{1}$.  Then $V$ is $\delta$-complete if and only if for every $K$-maximal $\delta$-ring $(R,f,\m)$ and point $p=(x:1)$ or $(1:x)\in V(K)$ we have either $x\in R$ or $\frac{1}{x}\in R$. 
\end{thm}

We now give the aforementioned alternative versions of the criterion. The proof of the next one uses arguments similar to the original and is omitted (see \cite{SimmonsThesis}).  
In addition to not mentioning maximal $\delta$-rings, this geometric version limits the cases that we need to check. The statement looks complicated, but it simply identifies possible witnesses of $\delta$-incompleteness. 

\begin{thm} 
Let $V$ be a differential subvariety of $\mathbb{P}^{n}$; call the projective coordinates $x_{0},\dots,x_{n}$. Let $\mathcal{W}$ be the collection of all differential subvarieties $W \subseteq \mathbb{P}^{n}\times \An{m}$  that have the following form: Choose a non-empty subset $I$ of $\{0,\dots, n\}$, and for each $i\in I$ choose some $j_{i}\neq i$ between $0$ and $n$. For each $i\in I$, choose  some finite number of $\delta$-monomials $t_{i_{k}}\!\!\left(\frac{x_{j_{i}}}{x_{i}}\right)$ whose variables have been replaced by the fraction $\frac{x_{j_{i}}}{x_{i}}$. Finally, for each $i\in I$ introduce new variables $y_{i_{k}}$ (there are $m$ such variables in all, representing the affine coordinates of $\An{m}$). Define $W$ to be the common zero locus in $\mathbb{P}^{n}\times \An{m}$  of the following $\delta$-polynomials:

\begin{itemize} 
 \item $x_{j}$, for all $j$ such that $0\leq j \leq n$ and $j\notin I$.
\item the $\delta$-homogeneous (with respect to $x_{0},\dots,x_{n}$) polynomials resulting from clearing denominators 
in $\sum_{k} y_{i_{k}}t_{i_{k}}\!\!\left(\frac{x_{j_{i}}}{x_{i}}\right) -1$, for all $i\in I$.  
\item The $\delta$-homogeneous $\delta$-polynomials that define $V$.
\end{itemize}

Then $V$ is $\delta$-complete if and only if for every $W\in\mathcal{W}$, the Kolchin closure $\overline{\pi_{2}(W)}$ of the image of the projection of $W$ into $\An{m}$ does not contain the point $\bar{0}$. 

\end{thm}

The third version requires the following definition and proposition that are essentially the same as for non-differential polynomial rings (see Theorem 3, p. 193, and Proposition 5, p. 397, of \cite{clo1}).

\begin{defn}
 Let $K\{\bar{x},\bar{y}\}=K\{x_{0},\dots, x_{n}, y_{1},\dots, y_{m}\}$ be a differential polynomial ring over a differential field $K$. If $I$ is a differential ideal of $K\{\bar{x},\bar{y}\}$ generated by $\delta$-polynomials $\delta$-homogeneous with respect to $x_{0},\dots,x_{n}$, we define the projective differential elimination ideal of $I$ eliminating $\bar{x}$ to be $\hat{I}_{\bar{x}}=\{ g\in K\{\bar{y}\} \mid \text{ for all } 0\leq i\leq n \text{ there is } e_{i} \geq 0 \text{ such that } x_{i}^{e_{i}}g\in I\}$.
\end{defn}

\begin{prop}
 If $K\models DCF_{0}$, then $\mathbf{V}(\hat{I}_{\bar{x}})=\overline{\pi_{2}(\mathbf{V}(I))}$.
\end{prop}

\begin{thm}
Let $V$, $\mathcal{W}$, and $W$ be as in the preceding theorem, and let $I$ be the differential ideal in $\mathcal{F}\{x_{0},\dots,x_{n}, y_{1},\dots,y_{m}\}$ generated by the defining polynomials of $W$.  Then $V$ is $\delta$-complete if and only if for all $W\in \mathcal{W}$ the projective differential elimination ideal $\hat{I}_{\bar{x}}$ contains a $\delta$-polynomial in $\bar{y}$ having non-zero constant term. (Here ``constant term" refers to any element of $\mathcal{F}$, not necessarily one having derivative equal to 0.) 
\end{thm}

\begin{proof}
If $V$ is $\delta$-complete, the preceding version of the projective valuative criterion asserts that $\bar{0}\notin \overline{\pi_{2}(W)}$. Because $\mathbf{V}(\hat{I}_{\bar{x}})=\overline{\pi_{2}(W)}$, $\hat{I}_{\bar{x}}$ must contain a polynomial with non-zero constant term. Conversely, if $V$ is not complete, then $\bar{0}\in \overline{\pi_{2}(W)}$ and so no member of $\hat{I}_{\bar{x}}$ can have a non-zero constant term.
\end{proof}

We now have three perspectives on $\delta$-completeness: the original projective valuative criterion with its use of maximal $\delta$-rings, a more geometric version using only the Kolchin closure of the image under projection of subvarieties of a special form, and a more ``syntactic" version concerning the $\delta$-polynomials that show up in elimination ideals. The latter two results give more concrete ways of showing that a specific $\delta$-variety is complete, but they are still difficult to work with given the well-known complexity of differential polynomial rings.

\end{section}

\begin{section}{New examples of $\delta$-completeness}
To our knowledge, the only complete $\delta$-varieties explicitly identified prior to the present work were Kolchin's example of the projective closure of the constants, Pong's example of the projective closure of $x'=P(x)$, and Freitag's extension of these examples to the partial differential case. In this section we list new $\delta$-complete examples we have found using the projective valuative criterion and explicit elimination algorithms. The following $\delta$-varieties are taken to lie in $\Pn{1}$. As before, $\cF$ is a model of $DCF_{0}$.

\begin{enumerate}
\item General classes:

\begin{itemize}
\item Projective closures of varieties defined by linear differential polynomials in one variable.

\item The projective closure of $x^{(n)} = P (x^{(n-1)})$ ($P$ a non-differential polynomial in $x^{(n-1)}$).

\item The projective closure of $P(x^{(n)})=0$ for any non-constant polynomial $P(x^{(n)})\in\mathcal{F}[x^{(n)}]$ (e.g., $(x^{(n)})^{2} - 3x^{(n)}+2=0$).
\end {itemize}

\item First-order: 
\begin{itemize}
\item The projective closure of $Q(x)x'=P(x)$ ($Q$ and $P$ non-differential polynomials in $x$). Note that this is a strict generalization of the previously known examples in one variable.
 \item The projective closure of $(x')^{n} = x+ \alpha$  ($\alpha$ is an arbitrary element of $\mathcal{F}$).
\end{itemize}

\item Second-order:
\begin{itemize}
\item The projective closure of $xx''=x'$.  This variety was shown by Poizat to have Morley rank 1 (see \cite{Marker:model_thy_of_fields} for the argument), so it has a lower model-theoretic rank than the order suggests. In spite of this anomaly, completeness shows that the variety at least is nice topologically.
\end{itemize}
\end{enumerate}•

We present proofs for two of these; for the others, consult the author's thesis (\cite{SimmonsThesis}).

\begin{thm}
\label{lincomplete}
The projective closure in $\Pn{1}$ of a $\delta$-variety defined by a linear ordinary differential equation in one variable is complete. 
\end{thm}

\begin{proof}
In $\An{1}$, such a variety is defined by an equation of the form $x^{(k)}=\sum_{i=0}^{k-1} a_{i}x^{(i)} + b$, for $a_{i}, b\in \mathcal{F}$. To invoke the projective valuative criterion we must show that for any $x\in K\models DCF_{0}$ satisfying the equation and $R$ a $K$-maximal $\delta$-ring, either $x$ or $\frac{1}{x}$ belongs to $R$. But this is exactly what Proposition 4  of \cite{MorrisonSD} says. (The promised use of explicit elimination is contained in Morrison's proof, which employs an argument similar in spirit to our methods.)
\end{proof}

Next we present a useful elimination algorithm that we use to show that the projective closure of $\Po$ is $\delta$-complete.  
\begin{prop}
\label{recip_elim_var}
Let $A$ be a ring (commutative with 1), $x,y$ units in $A$, and $B$ a subring (not necessarily containing $1$) of $A$. Suppose 1 satisfies the following equations in which the expressions on the right have coefficients from $B$:
\begin{enumerate}
\item$1= \sum b_{vw}y^{v}/x^{w} + \sum b_{z}/x^{z}$ where $0<v<w$ and $z$ is non-negative.
\item$1= \sum c_{rs}x^{r}/y^{s} + \sum c_{q}/x^{q}$ where $r<s$, $s>0$ and $q$ is non-negative.
\end{enumerate}•
Then $1\in B\left[\nfracx\right]$.
\end{prop}

\begin{proof}
We refer to expressions having the form of the first equation as \emph{type I equations}; monomials satisfying the corresponding conditions are \emph{type I terms}. The label \emph{type II} analogously applies to the second equation and its terms. We are only concerned with preservation of these conditions, so we do not need to track each coefficient or write all indices. In keeping with this, we use a placeholder symbol $c$ for the coefficients; this does not mean that they are identical, but simply that some coefficient from $B$ is present. Similarly, once we observe that a condition is preserved, we may reuse notation for terms of the same form. For example, rather than expanding $\left(1+\sum c/x^{q}\right)\left(1+\sum c/x^{q}\right)$, we simply write $1+\sum c/x^{q}$ again to represent the product, which is a sum of 1 and terms having coefficients from $B$ and a non-negative power of $x$ in the denominator.

Let $u$ be the maximal exponent of $y$ (in the numerator) in the type I equation, and $t$ the maximal exponent of $y$ (in the denominator) in the type II equation. If there is no $y$ in an equation, then respectively $u$ or $t$ is 0. There are two cases:
\begin{enumerate}
\item $u\leq t$: Subtract the $y$-free terms in the type I equation from both sides and divide by $y^{u}$ to obtain $(1 + \sum c/x^{z})/y^{u} = \sum cy^{v-u}/x^{w} = \sum c/ x^{w}y^{u-v}$. Multiply the type II equation by $(1 + \sum c/x^{z})$. Note that multiplying a type II term by $1+ \sum c/x^{z}$ preserves the type II requirements on the exponents of $x$ and $y$. Hence we may reuse the notation for the terms $cx^r/y^s$ such that $s<u$. Substituting for  $(1 + \sum c/x^{z})/y^{u}$ we obtain 
\begin{align*}
1&= \sum c/x^{q} + \sum_{0<s< u} cx^{r}/y^{s} + \left(\sum_{u\leq s}cx^{r}/y^{s-u}\right)\left(\sum  c/ x^{w}y^{u-v}\right) \\ &=  \sum c/x^{q} + \sum_{0<s< u} cx^{r}/y^{s} + \sum_{u\leq s} \sum cx^{r-w}/y^{s-v}
\end{align*}

\noindent It remains to verify the requirements on the exponents of terms for which $u\leq s$. For those, $v<w$ and $r<s$, so $r-w < s-v$. If $s=v$, the term becomes $y$-free with $x$ in the denominator. Otherwise, because $0< v$ the exponent of $y$ decreases but remains positive. Therefore we have a new type II equation but with a smaller maximal exponent of $y$ (the exponents of $y$ in $\sum_{0<s<u} cx^{r}/y^{s}$ are less than $u$ and hence than $t$). Thus after finitely many iterations there will no longer be an $s$ such that $u\leq s$ and we are guaranteed to move to the second case.  

\item $t< u$: Subtract the $y$-free terms in the type II equation from both sides and multiply by $y^{t}$ to obtain $(1+ \sum c/x^{q})y^{t} = \sum cx^{r}y^{t-s}$. Multiply the type I equation by $(1+ \sum c/x^{q})$. As before, this multiplication does not disturb the required balance of exponents and so we only alter the notation for terms divisible by $y^{t}$. We find 
\begin{align*}
1&= \sum c/x^{z} + \sum_{0<v<t}c y^{v}/x^{w} + \left(\sum_{t\leq v} cy^{v-t}/x^{w}\right)\left(\sum cx^{r}y^{t-s}\right) \\ &= \sum c/x^{z} + \sum_{0<v<t} cy^{v}/x^{w} + \sum_{t\leq v}\sum cy^{v-s}/x^{w-r}
\end{align*}

\noindent For the terms such that $t\leq v$, the exponent of $y$ is still non-negative because $s\leq t \leq v$. Again, $v<w$ and $r<s$, so $v-s< w-r$. Since $s>0$, we have a new type I equation with a lower maximal exponent of $y$ (the exponents of $y$ in $\sum_{0<v<t} y^{v}/x^{w}$ are less than $t$ and hence than $u$). Consequently, after finitely many iterations, there will no longer be a $v$ such that $t<v$ and we return to the first case.
\end{enumerate}

At the end of each step, we still have equations of type I and II. Because either $t$ or $u$ decreases with each iteration of the algorithm, either $u$ or $t$ eventually becomes 0. (The types do not allow $u$ or $t$ to be negative.) If $u$ or $t$ is 0, then in the remaining terms $x$ only appears in the denominator. This proves $1\in B[1/x]$.
 
\end{proof}

\begin{prop}
The projective closure of $\Po$ is $\delta$-complete.
\end{prop}

\begin{proof}
We use the projective valuative criterion for $\Pn{1}$. Let $(R,\m)$ be a $K$-maximal $\delta$-ring and let $p=(x:1)\in V(K)$, where $V$ is the projective closure of $\Po$. We must show that either $x\in R$ or $\nfracx\in R$. This is immediate if $x=0$; otherwise, if $x'=0$ then Morrison's result implies that either $x\in R$ or $\nfracx \in R$. So suppose that $x,x'$ are non-zero and also that neither $x$ nor $\nfracx$ belongs to $R$. As a result, we know that both $1\in\mx$ and $1\in \monex$.

We focus on $1\in \monex$. Because $x$ and $x'$ are non-zero, we may divide by both of them, as well as substitute $\nfracxpx$ for $x''$. Differentiating a monomial $(x')^{j}/x^{i}$ such that $i>j\geq0$ results in a sum of monomials such that each is a multiple of $(x')^{j}/x^{i}$ times either $x'/x$ (when differentiating $1/x$) or $1/x$ (when differentiating $x'$ and applying the substitution). This preserves the excess of $x$ in the denominator. As $\m$ is a local ring, we may always subtract terms that are elements of $\m$ from 1 and divide (we continue to do so without comment for any free elements of $\m$ that appear during the upcoming elimination). Hence we may assume that there are coefficients $c_{ij}\in \m$ and pairs of exponents $i>j\geq 0$ such that $1=\sum c_{ij}(x')^{j}/x^{i}$. Note that this satisfies the definition of a type I equation (with $x'$ playing the role of $y$) in the preceding proposition. Our strategy is to show that modulo differentiation and substitution of $\nfracxpx$ for $x''$, this equation implies $1\in \m\left[\nfracx\right]$ (and hence $x$ is integral over $R$). That will provide a contradiction because we know from Morrison that $R$ is integrally closed, so $x\in R$.

We eliminate $x'$ and show that the resulting relation $1\in \m[x,\nfracx, \frac{1}{x'}]$ is witnessed by a type II equation. Then the elimination algorithm of Proposition \ref{recip_elim_var} applied to the original type I equation and the resulting type II equation will imply that $1\in \m\left[\nfracx\right]$, as desired. 
 
Differentiate $1=\sum c_{ij}(x')^{j}/x^{i}$ and invoke the substitution to get
\begin{align*}
0=\sum c_{ij}'(x')^{j}/x^{i} +\sum(-i)c_{ij}(x')^{j+1}/x^{i+1} + \sum (j)c_{ij}(x')^{j}/x^{i+1}.
\end{align*}

\noindent We target the term $c_{IJ}(x')^{J}/x^{I}$ of the original equation that is maximal with respect to the monomial ordering $lex(\text{deg}(x'),\text{deg}(x))$, where the degrees of $x',x$ are taken with respect to $x'$ in the numerator and $x$ in the denominator. That is, first eliminate the term with the highest degree $I$ of $x$ in the denominator out of those terms having the highest overall degree $J>0$ of $x'$ in the numerator. (If  $x'$ does not appear, $1\in\m\left[\nfracx\right]$ already.). Do this by adding the following equation to the original:
\begin{align*}
0&=\left(\frac{x}{Ix'}\right)\left(\sum c_{ij}'(x')^{j}/x^{i} +\sum(-i)c_{ij}(x')^{j+1}/x^{i+1} + \sum (j)c_{ij}(x')^{j}/x^{i+1}\right)\\
&=\sum \frac{c_{ij}'}{I}(x')^{j-1}/x^{i-1} +\sum\frac{-i}{I}c_{ij}(x')^{j}/x^{i} + \sum \frac{j}{I}c_{ij}(x')^{j-1}/x^{i}.
\end{align*}  

\noindent We carefully examine the consequences of this elimination step.
\begin{itemize}

\item Claim: 1 is preserved. Confirmation: We added 0 to the original equation, and 1 does not appear among the added terms.

\item Claim: $c_{IJ}(x')^{J}/x^{I}$ was eliminated. Confirmation: Since $I>J>0$, dividing by $I$ was valid and the summands included the canceling term $-c_{IJ}(x')^{J}/x^{I}$.

\item Claim: Every remaining term is strictly less than $c_{IJ}(x')^{J}/x^{I}$ with respect to $lex(\text{deg}(x'),\text{deg}(x))$. Confirmation: The remaining terms from the original equation were already less than $c_{IJ}(x')^{J}/x^{I}$ in this ordering, and they didn't change. The new terms that were added all saw the exponent of $x'$ decrease, except for $\frac{-i}{I}c_{ij}(x')^{j}/x^{i}$. But these stayed the same as before, and so for $i\neq I$ are still less than $c_{IJ}(x')^{J}/x^{I}$.  In particular, among the new terms there is only one copy of the monomial $(x')^{J}/x^{I}$, so after cancellation of that term only lesser ones remain. Note that relations leading to cancellations among the other terms cause no problem; they can only hasten the elimination of non-type-II terms.    

\item Claim: All exponents of $x$ have stayed the same except those in the terms $\frac{c_{ij}'}{I}(x')^{j-1}/x^{i-1}$. Confirmation: This is immediate from the expression. The important thing to note is that the exponent of $x'$ in the numerator has decreased by the same amount. Since $0\leq j<i$, repetition of this operation will place $x'$ in the denominator before it places $x$ in the numerator. In other words, the descendants of $c_{ij}(x')^{j}/x^{i}$ that have $x$ in the numerator must have a larger power of $x'$ in the denominator; i.e., they will be type II terms. (Note that $c/(x^{i}(x')^{j})$ for $i,j$ non-negative and $c\in \m$ is type II; we don't need to have a positive power of $x$ in the numerator.)

Repeat until all terms with $x'$ in the numerator are eliminated. This process terminates because $lex(\text{deg}(x'),\text{deg}(x))$ is a well-ordering. Since 1 is preserved at each stage and the terms that are left are all type II terms, we have a type II equation. This completes the proof. 

\end{itemize}
\end{proof}

\end{section}

\begin{section}{Incomplete finite-rank differential varieties}
Success with the examples in the preceding section makes it tempting to think that all finite-rank projective $\delta$-varieties over a model of $DCF_{0}$ are $\delta$-complete. Such a result would be satisfying because then infinite rank would be the only obstruction distinguishing differential completeness from the algebraic case.  However, the $\delta$-completeness issue turns out to be more complex than that. The prevalence of incomplete, finite-rank $\delta$-varieties is still not clear, but we have found one family. 
(Recall that the \emph{separant} $s_{p}$ of a differential polynomial $p\in \mathcal{F}\{x\}$ is defined to be the formal partial derivative of $p$ with respect to the highest-order derivative appearing in $p$. For example, the separant of $3x(x')^{2} + x' -x$ is $6xx'+1$.)

\begin{prop}
The projective closure of $x''=x^{n}$ is not $\delta$-complete for $n\geq 2$.
\end{prop}
\begin{proof}
Let $\cF\models DCF_{0}$ and let $n\geq 2$. Consider the subset $W$ of $\An{1}(\cF)\times\An{1}(\cF)$ defined by $x''=x^{n}$ and $y(2x^{n+1}-(n+1)(x')^{2})=1$. We claim that $\pi_{2}(W)=\{y\mid y'=0 \text{ and } y\neq 0\}$. That set is not $\delta$-closed because $y'=0$ is an irreducible $\delta$-variety containing the point 0. (Note that $y'$ is irreducible as a polynomial. We use the notation $[f]$ to represent the differential ideal generated by the $\delta$-polynomial $f$.  By Corollary 1.7 (p. 44) of \cite{Marker:model_thy_of_fields}, it follows that the differential ideal $I(y')= \{g \in  \cF\{x\}\mid s_{y'}^{k}g\in [y'] \text{ for some }k\}$ is prime. The separant $s_{y'}$ of $y'$ is 1, so $I(y')=[y']$ is prime and $\mathbf{V}([y'])=\mathbf{V}(y'=0)$ is an irreducible $\delta$-variety.) This suffices to prove incompleteness of the projective closure because $x''=x^{n}$ is already projective if $n\geq 2$. (The proof for the case $n=2$ is in the appendix. For $n\geq 3$, note that the point at infinity does not lie on the $\delta$-homogenization of $x''=x^{n}$.)

For the first containment, suppose $y\in \pi_{2}(W)$. Let $(x,y)\in W$, and differentiate the equation $y(2x^{n+1}-(n+1)(x')^{2})=1$. Substituting $x^{n}$ for $x''$, we find $y'(2x^{n+1}-(n+1)(x')^{2}) + y(2(n+1)x^{n}x'-2(n+1)x'x^{n})=  y'(2x^{n+1}-(n+1)(x')^{2})= 0$. Multiplying both sides by $y$ and applying the relation $y(2x^{n+1}-(n+1)(x')^{2})=1$ gives $y'=0$; clearly $y\neq 0$.

For the other containment let $y$ be a non-zero element of $\mathcal{F}$ such that $y'=0$. We need to find $x$ such that $x''=x^{n}$ and $2x^{n+1}-(n+1)(x')^{2}=\frac{1}{y}$. Blum's axioms for $DCF_{0}$ imply that there exists $x$ such that simultaneously $2x^{n+1}-(n+1)(x')^{2}=\frac{1}{y}$ and $2x^{n+1}\neq \frac{1}{y}$. Therefore $x'\neq 0$. Differentiation produces $2(n+1)x^{n}x'-2(n+1)x'x''=0$. Because $x'\neq 0$, we may divide and obtain $x''=x^{n}$. Hence $y\in \pi_{2}(W)$.  
\end{proof}

In retrospect, this example has a similar flavor to Kolchin's example proving incompleteness of $\Pn{1}$. Both instances feature differentiation of polynomials having constant term equal to 1 as well as the use of Blum's axioms to prevent the separant from vanishing. The non-closed image in our example has finite rank, though, because of the additional constraint imposed by the relation $x''=x^{n}$. Thus not only is having too few relations (i.e., infinite rank) deleterious to completeness, but having too many of the wrong kind is also an obstacle. To illustrate, consider the example in the case $n=3$. The image is not closed because the derivative of $2x^{4}-4(x')^{2}$ happens to belong to the $\delta$-ideal $[x''-x^{3}]$, though $2x^{4}-4(x')^{2}$ itself does not. 

\end{section}

\begin{section}{Appendix}

\begin{prop}
The projective closure in $\mathbb{P}^1$ of the affine $\delta$-variety defined by $x''=x^2$ is the affine variety itself. \end{prop}
\begin{proof}
We claim that the projective closure of $x''=x^{2}$ in $\Pn{1}$ is simply the set $V=\{(x:y)\in \Pn{1}\mid (x:y)\in \mathbf{V}((x''-x^{2})^{\delta h}) \text{ and } y\neq 0\}$, where the $\delta$-homogenization $(x''-x^{2})^{\delta h}$ of $x''-x^2$ is
\[
y^{2}x'' -2yy'x' -yy''x + 2(y')^{2}x -yx^{2}.
\] 

\noindent That is, the affine $\delta$-variety $x''=x^{2}$ contained in the first chart is already projectively closed. This is counterintuitive because the point at infinity $(1:0)$ satisfies the equation $(x''-x^{2})^{\delta h}=0$. (Such a thing is possible because the $\delta$-projective closure is defined by the $\delta$-homogenization of the entire $\delta$-ideal $[x''-x^2]$, not just the generator. Similar problems arise even in the algebraic case; e.g., Example 3, p. 387 in \cite{clo1}.) 

Elementary topology implies (see, e.g., Lemma 1, p. 49, \cite{shaf}) that $V$ is closed in $\Pn{1}$ if and only if its affine restrictions are closed in the subspace Kolchin topology on both standard affine charts. This is obviously the case for the first chart; set $y=1$, and we get the original equation $x''=x^{2}$. It is not obvious for the second chart because we have insisted on excluding the point $(1:0)$. Indeed, letting $x=1$ in the equation $y^{2}x'' -2yy'x' -yy''x + 2(y')^{2}x -yx^{2}=0$ we obtain $-yy'' +2(y')^{2} -y=0$; it is not clear that we may remove the solution $y=0$ and still have a Kolchin-closed affine set. Nonetheless, we will find $\delta$-polynomials whose common zero locus is precisely the set $-yy'' +2(y')^{2} -y=0\wedge y\neq 0$.

Consider the $\delta$-elimination ideal $I_{t}$ of the differential ideal $I=[-yy'' +2(y')^{2} -y, yt-1]$. Because we are working over $DCF_0$, $I_{t}$ defines the Kolchin closure of $-yy'' +2(y')^{2} -y=0\wedge y\neq 0$ in $\mathbb{A}^1$. We must show that the closure adds no new points. We claim that the affine restriction $-yy'' +2(y')^{2} -y=0\wedge y\neq 0$ is equivalent to the closed condition $-y''y+2(y')^{2}-y=0 \wedge 3(y'')^{2}-1-2y'''y'+2y''=0$.  (To obtain these equations we differentiated $-y''y+2(y')^{2}-y$, treated each derivative as an algebraic indeterminate, and used Maple's {\tt{EliminationIdeal}} command to find generators of $I_{t}$ of up to order 3. The next paragraph shows that the resulting equations justify the claim.) 

Let $p_{1}=3(y'')^{2}-1-2y'''y'+2y'', p_{2}=-y''y+2(y')^{2}-y,$ and $p_{3}=p_2'=3y'y''-yy'''-y'$. Observe that $yp_{1}=(-3y'' +1)p_{2}+ 2y'p_{3}$. This proves that the solutions of the system $-y''y+2(y')^{2}-y=0 \wedge 3(y'')^{2}-1-2y'''y'+2y''=0$ indeed contain the set  $-yy'' +2(y')^{2} -y=0\wedge y\neq 0$. The reverse containment follows from the presence of the term $-1$ in $p_{1}$.

 \end{proof}
\end{section}
\bibliography{simmons_thesis_mining_copy_updated}
\bibliographystyle{plain}

\end{document}